\documentclass[11pt]{amsart}
\usepackage{amsfonts}
\usepackage{amsmath}
\usepackage{amssymb}
\usepackage{color}
\usepackage[all]{xy}
\usepackage{graphicx}
\usepackage{xspace}

\input{xy}
\xyoption{all}

\newcommand{\sym}{\mathfrak{S}}
\renewcommand{\S}{\mathcal{S}}
\newcommand{\VS}{\mathbf{S}}
\newcommand{\Descd}{\mathbf{Descd}}
\newcommand{\tdelta}{\tilde{\Delta}}

\newcommand{\CC}{\mathbf{C}}
\newcommand{\B}{\mathcal{B}}

\def\shuff#1#2{\mathbin{
      \hbox{\vbox{
        \hbox{\vrule
              \hskip#2
              \vrule height#1 width 0pt
               }%
        \hrule}%
             \vbox{
        \hbox{\vrule
              \hskip#2
              \vrule height#1 width 0pt
               \vrule }%
        \hrule}%
}}}

\def\shuffl{{\mathchoice{\shuff{7pt}{3.5pt}}%
                        {\shuff{6pt}{3pt}}%
                        {\shuff{4pt}{2pt}}%
                        {\shuff{3pt}{1.5pt}}}}%
\def\shuffle{\, \shuffl \,}

\newtheorem{defi}{\indent Definition}
\newtheorem{lemma}[defi]{\indent Lemma}
\newtheorem{cor}[defi]{\indent Corollary}
\newtheorem{theo}[defi]{\indent Theorem}
\newtheorem{prop}[defi]{\indent Proposition}

\begin{document}

\title[Structure of shuffle algebras]{\Large{Natural endomorphisms of shuffle algebras}}

\author{Lo\"ic Foissy}

\author{Fr\'ed\'eric Patras}

\date{}

\maketitle

\section{Introduction}

Shuffles have a long history, starting with the probabilistic study of card shufflings in the first part of the 20th century by Borel, Hadamard, Poincar\'e and others. Their theory was revived in the 50's, for various reasons. In topology, the combinatorics of (non commutative) shuffle products was the key to the definition of topological products such as the ones existing on cochain algebras and the cohomology groups of topological spaces. Commutative shuffle products were the key to the study of the homology of abelian groups and commutative algebras. In combinatorics and for the theory of iterated integrals, commutative shuffle products played a key role 
resulting in the global picture of the modern theory of free Lie algebras given in C. Reutenauer's seminal \it Free Lie algebras \rm \cite{reutenauer}.

The classical approach to shuffle algebras, as featured for example in Reutenauer's book, focussed on Lie theoretical properties, that is on the enveloping algebra structure of tensor algebras: the shuffle product arises naturally in this framework by dualizing the Hopf algebra structure of the tensor algebra and many properties of shuffles can be derived from that particular approach.

However, one can try to follow a different path, namely start directly from the combinatorics of shuffles, following the ideas originally developed by M.-P. Sch\"utzenberger \cite{schutz}. A series of recent works by F. Chapoton, C. Malvenuto, C. Reutenauer, the second author of the present article, and others, provides many new tools to revisit the theory of shuffles. This is the purpose of the present article to put these tools to use. 

Concretely, we focus on the adaptation to the study of shuffles of the main combinatorial tool in the theory of free Lie algebras, namely the existence of a universal algebra of endomorphisms for tensor and other cocommutative Hopf algebras: the family of Solomon's descent algebras of type $A$ \cite{reutenauer,patJA}.
We show that there exists similarly a natural endomorphism algebra for commutative shuffle algebras, which is a natural extension of the Malvenuto-Reutenauer Hopf algebra of permutations, or algebra of free quasi-symmetric functions. We study this new algebra for its own, establish freeness properties, study its generators, bases, and also feature its relations to the internal structure of shuffle algebras.

\section*{Acknowledgements} We thank warmly C. Reutenauer for several discussions in Montr\'eal that originated the present work, and the UQAM (Universit\'e du Qu\'ebec \`a Montr\'eal) and the LIRCO (Laboratoire International Franco-Qu\'eb\'ecois de Recherche en Combinatoire) for their support.

\section{Shuffles}

As mentioned in the introduction, shuffle products can be understood in the commutative and noncommutative frameworks. The two uses still coexist (topological shuffles are noncommutative, whereas the use in combinatorics is to refer to shuffle products as the commutative ones of the theory of free Lie algebras). We survey briefly the historical foundations of the theory since it will appear later that the combinatorics of the objects that were first considered to study shuffle products (geometrical simplices and tensors) is closely related to the new algebraic structures to be introduced in the present paper.

According to \cite{McL1}, the algebraic theory of these products was first established in \cite{EM} together with the introduction of the notion of half-shuffles, which was to become the classical way to define recursively shuffle products.
In view of later developments, the fundamental observation \cite[Fla 5.7]{EM} is that the topological (cartesian) product of simplices $\ast$ decomposes into two half-products $\prec , \succ $:
$$x\ast y = x \prec y +x\succ y.$$ The associativity of the product follows then formally from the distributivity of left ($\prec $) and right ($\succ $) half-shuffles with respect to the $\ast$ product \cite[Thm 5.2]{EM}. 

Whereas the topological product is associative but not commutative (the associativity holds automatically, but the commutativity holds only up to orientation and homotopy), Eilenberg and MacLane were the first to consider also the purely commutative case when dealing with the bar construction (a topological object which combinatorial structure is the one of the tensor algebra), see \cite[Sect. 18]{EM} and \cite{EM0}. These ideas were rediscovered independently by M.P. Sch\"utzenberger \cite[1-18]{schutz}, who also clarified the set of relations necessary to prove the associativity relation (as was realized later, his proof of the associativity relation does not require the commutativity asumption and, up to a rewriting, coincides in the end essentially with the one given by Eilenberg-MacLane in a topological framework). With our previous notation, the half-shuffles associativity relations read:
\begin{eqnarray}\label{dend}
x\prec (y\ast z)&=&(x\prec y)\prec z; \\
\label{dend2} (x\ast y)\succ z &=&x\succ (y\succ z); \\
\label{dend3}  (x\succ y)\prec z&=&x\succ (y\prec z). 
\end{eqnarray}
In the commutative case, the commutativity property translates into $x\prec y = y\succ x$ and these relations simplify to
\begin{equation}\label{shuffle}
(a\prec b)\prec c=a\prec (b\prec c+c\prec b), 
\end{equation}
see \cite[(18.7)]{EM}, \cite[(S0)]{schutz}, \cite{lod}.

For simplicity, we stick from now on to the current terminology and call \it shuffle algebra \rm a commutative shuffle algebra, that is an algebra with a non associative ``half-product'' $\prec $ satisfying the relation (\ref{shuffle}) and \it dendriform algebra \rm a noncommutative shuffle algebra, that is an associative algebra with two half-products satisfying the associativity relations (\ref{dend}-\ref{dend3}) (but not the commutativity relation $x\prec y=y\succ x$; the half-shuffles relations have also been attributed to Rota, see \cite{lod}). See also \cite{Chap1, Aguiar, Foissy1, km, nt2}
for further general informations on the subject and applications of dendriform structures to various problems in algebra and combinatorics related to the ones we consider in the present article.

Shuffle algebras are sometimes refered to as Zinbiel algebras as a follow up of Cuvier's Jan. 1991 Thesis where Leibniz algebras were first introduced and studied (Zinbiel is the word Leibniz inverted, a successful joke suggested by the topologist J.M. Lemaire). Cuvier proved indeed that the cochain complex computing the homology of Leibniz algebras is the tensor algebra \cite{cuv1,cuv2} -from which one can deduce by standard procedures that the notions of Leibniz algebras and shuffle algebras are Koszul dual \cite{GK}. Since the original name ``alg\`ebres de shuffle'' is better known and accepted we prefer to stick to the usual terminology.

The classical shuffle bialgebra over an alphabet fits into this picture. 
Let $X$ be a graded, connected alphabet, that is to say $\displaystyle X=\bigcup_{n \geq 1} X_n$. 
For all $x \in X_n$, we put $|x|=n$: this is the \emph{weight} of $x$.
Let $T(X)$ be the tensor algebra generated by $X$ over $\mathbf{Q}$. For all $n \in \mathbb{N}$, let $T_n(X)$ be the subspace of $T(X)$ generated
by the words $y_1...y_n,\ y_i\in X$ of length $n$, and $T^n(X)$ be the subspace generated by the words of weight $n$, the weight of a word being the sum of the weights of its letters: $|y_1...y_n|:=|y_1|+...+|y_n|$. The product in the tensor algebra (the concatenation product) is written $\cdot$:
$$y_1...y_n \cdot z_1...z_p:= y_1...y_nz_1...z_p.$$

\begin{defi} 
The shuffle bialgebra $Sh(X)=\bigoplus\limits_{n\in\mathbb{N}}Sh^n(X) $ is the graded connected (i.e. $Sh^0(X)=\mathbf{Q}$) commutative Hopf algebra such that
\begin{itemize}
 \item The component of degree $n$ of $Sh(X)$, $Sh^n(X)$ is the linear span of the words of weight $n$ over $X$ (so that as vector spaces $Sh^n(X)=T^n(X)$). We write similarly $Sh_n(X)$ for the linear span of the words of length $n$;
\item The product $\shuffle$ is defined recursively as the sum of the two half-shuffle products $\prec ,\ \succ $:
$$y_1...y_n \prec  z_1...z_p := y_1\cdot (y_2...y_n \prec  z_1...z_p)$$
with $\shuffle = \prec  + \succ $ and $y_1...y_n \succ  z_1...z_p := z_1...z_p \prec  y_1...y_n$.
\item The coalgebra structure is defined by the deconcatenation coproduct:
$$\Delta (y_1...y_n):=\sum\limits_{0\leq k\leq n} y_1...y_k\otimes y_{k+1}...y_n.$$
\end{itemize}
\end{defi}

Recall that the notions of connected commutative Hopf algebra and connected commutative bialgebra are equivalent since a graded connected commutative bialgebra always has an antipode. The (graded) dual bialgebra of $Sh(X)$ is the tensor algebra $T(X)$ over $X$, we refer to \cite{reutenauer} for details and proofs.

Equivalently, for all $x_1,\ldots,x_{k+l} \in X$:
\begin{eqnarray*}
x_1\ldots x_k \prec x_{k+1}\ldots x_{k+l}&=& \sum_{\alpha\in Des_{\subseteq \{k\}},\:\alpha^{-1}(1)=1}
x_{\alpha^{-1}(1)}\ldots x_{\alpha^{-1}(k+l)},\\
x_1\ldots x_k \succ x_{k+1}\ldots x_{k+l}&=& \sum_{\alpha\in Des_{\subseteq \{k\}},\:\alpha^{-1}(1)=k+1}
x_{\alpha^{-1}(1)}\ldots x_{\alpha^{-1}(k+l)},
\end{eqnarray*}
where the $\alpha$ are permutations  of $[k+l]=\{1,...,k+l\}$.

The notation $\alpha\in Des_{\subseteq \{k\}}$ means that $\alpha$ has at most one descent in position $k$. 
Recall that a permutation $\sigma$ of $[n]$ is said to have a descent in position $i<n$ if $\sigma(i)>\sigma(i+1)$. The descent set of $\sigma$, $desc(\alpha)$ is the set of all descents of $\alpha$,
$$desc(\sigma):=\{i<n ,\sigma(i)>\sigma(i+1)\}.$$
For $I\subset [n]$, we write $Des_{I}:=\{\sigma , desc(\sigma)=I\}$ and $Des_{\subseteq I}:=\{\sigma , desc(\sigma)\subseteq I\}.$

\begin{prop}\label{freeness}
 As a commutative algebra, $Sh(X)$ is the \it free algebra \rm over $X$ for the relations (\ref{shuffle}). 
\end{prop}

The result goes back to \cite{schutz}, where the reader can also find a discussion of the role of the unit in shuffle algebras (there is a subtelty to make sense of the half-products with 1, however this problem is easily settled and doesn't need to be discussed here: we will only use the fact that $1$ is a unit for $\shuffle$ and, for half-shuffle products of 1 with words $w$ use Sch\"utzenberger's conventions $w\prec 1=w,\ 1\prec w=0$).

Whereas most studies focussed on shuffle algebras over sets,
shuffle algebras over graded sets are equally important objects. Two classical examples are provided by the iterated bar construction (a key to the computation of the homology of $K(\Pi , n)$ spaces \cite{EM0,EM}) and \it mould calculus \rm, which focusses on problems such as the study and classification of differential equations by algebraic means \cite{sauzin,mnt}. 
In the first framework, one constructs the shuffle algebra over a graded commutative algebra (this is actually one of the reasons for Eilenberg and MacLane works on shuffles), in the second case, the shuffle algebra over graded derivations (e.g., in dimension 1, the family of the $x^n\partial_x$ with degree $n-1$).

The remaining part of the present article is devoted to the internal study of $Sh(X)$, where $X$ is a graded alphabet. We insist on the action of natural endomorphisms, mimicking what is known for the shuffle algebra over a non graded set. 
We also recover as a byproduct Chapoton's rigidity theorem showing that an abstract shuffle bialgebra (an abstract shuffle algebra with a suitable coproduct) can always be realized as the shuffle algebra over a graded set \cite{Chap1}.

\section{Graded permutations}

We have mentioned the foundational connexion between shuffles and the geometry of simplices. This relationship can be encoded purely combinatorially by the existence of a noncommutative shuffle product on the direct sum of the symmetric group algebras, this is the ``geometrical ring of the symmetric groups'' of  \cite[p. 180]{patSMF}, a construction that relates directly the classical geometrical approach to shuffle products with the combinatorial approach.

The direct sum of the symmetric group algebras $\bf S$ carries in fact a much richer structure than a mere dendriform product: Malvenuto and Reutenauer first showed that it carries actually a noncommutative noncocommutative Hopf algebra structure and proved that it generalizes naturally various fundamental algebraic structures in the theory of free Lie algebras such as Solomon's descent algebras or quasi-symmetric functions, two noncommutative generalizations of the ring of symmetric functions \cite{MR}. This Hopf algebra or Malvenuto-Reutenauer (MR) Hopf algebra can be furthermore realized as an algebra of generalized quasi-symmetric functions and is often refered to in the litterature as the Hopf algebra of free quasi-symmetric functions \cite{dht}. 

The MR Hopf algebra is closely related to various fundamental notions of noncommutative representation theory such as the descent algebra of type $A$ or the algebra of quasi-symmetric functions. These later notions are known to generalize to other Coxeter groups than the symmetric groups and, up to a certain extent, to wreath-products of symmetric groups with cyclic groups \cite{mr2,bh}. Colored permutations appear naturally in this framework (the finite set of colors corresponding to the elements of the cyclic groups) \cite{nt}. Some of our results generalize further these results from the case of finite cyclic groups to the integers.

These new Hopf algebras are typical examples of ``combinatorial quantum groups'' (graded Hopf algebras which are neither commutative nor cocommutative but can be naturally interpreted as a ``group of symmetries'', e.g. through the natural action of permutations on tensors) and have originated many studies. The one we will focus on and generalize is due to the second Author of the present article, who introduced the notion of bidendriform bialgebra and showed that the MR Hopf algebra carries such a structure -with various consequences such as the proof of the Free Lie conjecture (according to which the primitive elements of the MR Hopf algebra form a free Lie algebra) \cite{Foissy1}. 
These results, as we show now, generalize to graded permutations, which are a natural extension of the notions of permutations and colored permutations when studying shuffle algebras over graded sets. 

\ \par

Recall first that, as for any Hopf algebra, $End(Sh(X))$ carries an associative convolution product $\star$ defined  by 
$$(f\star g)=\shuffle \circ (f\otimes g)\circ \Delta.$$ 
We define two other products on $End(Sh(X))$ by $$f\prec g=\prec \circ (f\otimes g)\circ \Delta$$
and $$f\succ g=\succ\circ (f\otimes g)\circ \Delta.$$ As $\prec+\succ=\shuffle$, $\star=\prec+\succ$.

\begin{lemma}\label{dendalgle}
$(End(Sh(X)),\prec,\succ)$ is a dendriform algebra.
\end{lemma}

\begin{proof} Let $f,g,h \in End(Sh(X))$. Then:
\begin{eqnarray*}
(f\prec g)\prec h&=&\prec \circ (\prec \otimes Id)\circ (f\otimes g \otimes h) \circ (\Delta\otimes Id)\circ \Delta\\
&=&\prec\circ (Id \otimes \shuffle)\circ (f\otimes g\otimes h)\circ (Id \otimes \Delta)\circ \Delta\\
&=&f\prec(g\star h).
\end{eqnarray*}
The two other axioms are proved in the same way. \end{proof}

These products on $End(Sh(X))$ dualize to the tensor algebra, for this dual point of view we refer to \cite{fisher}, which contains various applications of dendriform structures to Hopf algebras of graphs and twisted Hopf algebras (Hopf algebras in the category of species).

The following Lemma, although a direct, straightforward, consequence of the recursive definition of the shuffle product on $Sh(X)$ and of the dendriform structure on $End(Sh(X))$ will prove very useful.

\begin{lemma}[Rewriting Lemma]\label{rewriting}
 For any $y_1,...,y_n$ in $X$, the word $y_1...y_n$ can be rewritten:
$$y_1...y_n=y_1\prec (y_2\prec (...\prec(y_{n-1}\prec y_n)...)).$$
\end{lemma}

The proof is left as an exercice. The Lemma can be used to prove the Prop.\ref{freeness}, see \cite[1.19]{schutz}.

\begin{cor}
If we write $\pi=\sum_{n\in\mathbb{N}^\ast}\pi_n$ the projection on $Sh_1(X)$ (where $\pi_n$ is the projection on $Sh_1(X)\cap Sh^n(X)$) orthogonally to the $Sh_n(X)$, $n\not= 1$, we get:
$$Id=\exp^{\prec}(\pi):=\sum_{n\in\mathbb N}\pi^{\prec n}=\sum_{n\in\mathbb N} \pi \prec (\pi\prec (...(\pi \prec \pi)...)),$$
where $\pi^{\prec 0}$ stands for the canonical projection on the scalars, $Sh_0(X)$, and $\pi^{\prec n}:=\pi\prec \pi^{\prec n-1}$.
\end{cor}

The series $\exp^{\prec}(\pi)$ is the ``time-ordered exponential'' of physicists; it is often called in analyis and physics the Picard series or Dyson-Chen series of $\pi$ (see e.g. \cite{bp}, where the link between these series and the Malvenuto-Reutenauer Hopf algebra is explained). Its structure was investigated recently in a series of articles, focussing mainly on the Magnus problem (find an expression for $\Omega:=\log(\exp^{\prec}(\pi))$), see \cite{emp,km}). We will be interested here in different issues, related to the meaning of the $\pi_n$ and their products with respect to the internal structure of shuffle algebras. Notice that, by their very definition and due to the Rewriting Lemma \ref{rewriting}, we have:

\begin{lemma}
 For any $n_1,...,n_k\in\mathbb N^\ast$, $\pi_{n_1,...,n_k}:=\pi_{n_1}\prec(\pi_{n_2}\prec ...(\pi_{n_{k-1}}\prec\pi_{n_k})...)$ is the canonical projection on the linear span of words $x_1...x_k$ with $|x_i|=n_i$. In particular, $\pi_{n_1,...,n_k}\circ\pi_{n_1,...,n_k}=\pi_{n_1,...,n_k}$, and the $\pi_{n_1,...,n_k}$ form a complete family (i.e. with total sum $Id$) of orthogonal idempotents in $End(Sh(X))$.
\end{lemma}

We are now in the position to define  and study the algebra of graded permutations.

\begin{defi}
Let us fix $k \in \mathbb{N}$. Let $\sigma \in \sym_k$ and $d:[k]\longrightarrow \mathbb{N}^*$.
We define a linear endomorphism of $Sh(X)$ by:
$$\Phi_{(\sigma,d)}:\left\{\begin{array}{rcl}
x_1\ldots x_l&\longrightarrow&x_{\sigma(1)}\ldots x_{\sigma(l)} \mbox{ if }k=l \mbox{ and }|x_{\sigma(i)}|=d(i) \mbox{ for all }i,\\
&\longrightarrow&0\mbox{ if not.}
\end{array}\right.$$
\end{defi}

For example, $\pi_n=\Phi_{(id_n,n)}$.\\

{\bf Notations.}\begin{enumerate}
\item We put $\displaystyle \S=\coprod_{k\geq 0} \sym_k \times Hom([k],\mathbb{N}^*)$, and $\VS=Vect(\S)$.
\item Let $\sigma \in \sym_k$ and $d:[k]\longrightarrow \mathbb{N}^*$.
We shall represent $(\sigma,d)$ by the biword $\left(\begin{array}{ccc}\sigma(1)&\ldots&\sigma(k)\\
d(1)&\ldots&d(k)\end{array}\right)$.
\end{enumerate}

\begin{lemma}
For all $(\sigma,d) \in \sym_k \times Hom([k],\mathbb{N}^*)$ and $(\tau,e) \in \sym_l \times Hom([l],\mathbb{N}^*)$,
$$\Phi(\sigma,d)\circ \Phi(\tau,e)=\left\{\begin{array}{l}
\Phi(\tau \circ \sigma,d)\mbox{ if }k=l \mbox{ and }d=e\circ \sigma,\\
0\mbox{ if not.}
\end{array}\right.$$
\end{lemma}

\begin{proof} Direct computation. \end{proof}

{\bf Remark.} If for all $n\geq 1$, $X_n$ is infinite, it is not difficult to show that the linear extension $\Phi:\VS\longrightarrow End(Sh(X))$ is injective.
Hence, we can define an associative internal product on $\VS$ by:
$$(\sigma,d)\circ (\tau,e)=\left\{\begin{array}{l}
(\tau \circ \sigma,d)\mbox{ if }k=l \mbox{ and }d=e\circ \tau,\\
0\mbox{ if not.}
\end{array}\right.$$
We shall from now on identify $\VS$ with a subspace of $End(Sh(X))$ via $\Phi$. \\

Let us write $p_n$ for the canonical projection on $Sh^n(X)$, so that $Id=\sum_np_n$. A direct inspection shows that the $p_n$ belong to $\VS$:
\begin{lemma}
For all $n \geq 0$:
$$p_n=\sum_{k=1}^n \sum_{\substack{p:[k]\longrightarrow \mathbb{N}^*\\p(1)+\ldots+p(k)=n}}(Id_k,p)$$
$$=\sum_{k=1}^n \sum_{d(1)+\ldots+d(k)=n} \left(\begin{array}{ccc}1&\ldots&k\\d(1)&\ldots&d(k)\end{array}\right).$$
\end{lemma}

{\bf Notations}. \begin{enumerate}
\item Let $\sigma \in \sym_k$, $\tau \in \sym_l$. We define $\sigma \otimes \tau \in \sym_{k+l}$ by
$(\sigma \otimes \tau)(i)=\sigma(i)$ if $1\leq i\leq k$ and $(\sigma \otimes \tau)(i)=\tau(i-k)+k$ if $k+1 \leq i \leq k+l$.
\item Let $d:[k] \longrightarrow \mathbb{N}^*$ and $e:[l]\longrightarrow \mathbb{N}^*$. We define $d \otimes e:[k+l]\longrightarrow \mathbb{N}^*$ by
$(d\otimes e)(i)=d(i)$ if $1\leq i \leq k$ and $(d \otimes e)(i)=e(i-k)$ if $k+1 \leq i \leq k+l$.
\end{enumerate}

\begin{lemma}
$\VS$ is a dendriform subalgebra of $End(Sh(X))$.
\end{lemma}

\begin{proof} Let $(\sigma,d)$ and $(\tau,e) \in \VS$. We assume that $\sigma \in \sym_k$ and $\tau \in \sym_l$.
If $n \neq k+l$, then $((\sigma,d)\prec (\tau,e))(x_1\ldots x_n)=0$. If $n=k+l$, then:
\begin{eqnarray*}
((\sigma,d)\prec (\tau,e))(x_1\ldots x_{k+l})&=&
(\sigma,d).(x_1\ldots x_k) \prec (\tau,e).(x_{k+1}\ldots x_{k+l})\\
&=&\left\{\begin{array}{l}
x_{\sigma \otimes \tau(1)}\ldots x_{\sigma \otimes \tau(k)}\prec
x_{\sigma \otimes \tau(k+1)}\ldots x_{\sigma \otimes \tau(k+l)}\\[2mm]
\hspace{1cm} \mbox{ if } |x_{\sigma \otimes \tau(i)}|=(d\otimes e)(i) \mbox{ for all }i  ,\\[3mm]
0\mbox{ if not,}
\end{array}\right. \\
&=&\left\{\begin{array}{l}
\displaystyle \sum_{\alpha\in Des_{\subseteq \{k\}},\:\alpha^{-1}(1)=1}
x_{(\sigma \otimes \tau)\circ \alpha^{-1}(1)}\ldots x_{(\sigma \otimes \tau)\circ \alpha^{-1}(k+l)}\\[2mm]
\hspace{1cm} \mbox{ if } |x_{(\sigma \otimes \tau)\circ \alpha^{-1}(i)}|=(d\otimes e)\circ \alpha^{-1}(i) \mbox{ for all }i  ,\\[3mm]
0\mbox{ if not.}
\end{array}\right. 
\end{eqnarray*}
Hence:
\begin{equation}\label{EQ1}
(\sigma,d)\prec (\tau,e)=\sum_{\alpha \in  Des_{\subseteq \{k\}},\alpha^{-1}(1)=1}
((\sigma \otimes \tau)\circ \alpha^{-1},(d\otimes e)\circ \alpha^{-1}).
\end{equation}
Similarly:
\begin{equation}\label{EQ2}
(\sigma,d)\succ (\tau,e)=\sum_{\alpha \in  Des_{\subseteq \{k\}},\alpha^{-1}(1)=k+1}
((\sigma \otimes \tau)\circ \alpha^{-1},(d\otimes e)\circ \alpha^{-1}).
\end{equation}
So $\VS$ is a dendriform subalgebra of $End(Sh(X))$. \end{proof}

\begin{lemma}
 The idempotents $\pi_{n_1,...,n_k}$ belong to $\VS$ and generate a commutative subalgebra thereof for the composition product.
\end{lemma}

\begin{proof} The second part of the proposition being a straigthforward consequence of the idempotency property, let us show that $\pi_n$ belong to $\VS$; 
since $\VS$ is a dendriform subalgebra of $End(Sh(X))$, the Lemma will follow.

Since $\pi_1=(id_1,1)$, let us assume that $\pi_i\in \VS$ for $i<n$. We get:
$$\pi_n=p_n-\sum_{i_1+...+i_k=n,\ k>1}\pi_{i_1,...,i_k}$$
and the Lemma follows by induction. \end{proof}

\begin{prop}\label{eqnPi}
 For all $n\geq 1$, we have:
$$\pi_n=\left(\begin{array}{c}1\\n\end{array}\right)=\sum_{k=1}^n(-1)^{k-1}\sum_{a_1+\ldots+a_k=n} p_{a_1}\prec(p_{a_2}\star\ldots \star p_{a_k}).$$
\end{prop}

\begin{proof} Indeed, we have, according to the Rewriting lemma (\ref{rewriting}):
$p=\sum_{n\in\mathbb N}\pi^{\prec n}$ or, equivalently:
$$(1-\pi)\prec p=\pi^{\prec 0}=p_0$$
or: $\pi\prec p=\sum_{n\in\mathbb N^\ast}p_n$. Let us set $p^+:=\sum_{n\in\mathbb N^\ast}p_n $ and write $z$ for the convolution inverse of $p$ in $End(Sh(X))$ (recall that $p_0$ is the identity for the convolution product): 
$$z=p^{-1}=\sum_{k\in\mathbb N}(-1)^{k}(p^+)^k,$$
we get (recall that according to our conventions, $p_0$ is a right unit for $\prec$): 
$$\pi = \pi\prec p_0 = \pi\prec (p\star z)=(\pi \prec p)\prec z$$
where the third identity follows from the half-shuffle relations, so that:
$$\pi = (\sum_{n\in\mathbb N^\ast}p_n)\prec z=(\sum_{n\in\mathbb N^\ast}p_n)\prec \sum_{n\in\mathbb N}(-1)^{n}(p^+)^n,$$
from which the Proposition follows. \end{proof} 

Notice that the same argument would prove the following Lemma, useful to study Magnus formulas and Picard/Dyson-Chen series:

\begin{lemma}
  For any formally invertible series $q=1+\sum_{n\in\mathbb N^\ast}q_i$ in a dendriform algebra, we have:
$$ \mu= (\sum_{n\in\mathbb N^\ast}q_n)\prec \sum_{k\in\mathbb N}(-1)^{k}(\sum_{n\in\mathbb N^\ast}q_n)^k,$$
where $\mu$ is given by: $q=\exp^{\prec}\mu$.
\end{lemma}

{\bf Remark.} From (\ref{EQ1}) and (\ref{EQ2}), we obtain that $(\sigma,d)\prec (\tau,e)$ is the sum of the shufflings of the biwords 
 $\left(\begin{array}{ccc}\sigma(1)&\ldots&\sigma(k)\\d(1)&\ldots&d(k)\end{array}\right)$
 and  $\left(\begin{array}{ccc}\tau(1)+k&\ldots&\tau(l)+k\\ e(1)&\ldots&e(k)\end{array}\right)$
 such that the first biletter is $\left(\begin{array}{c} \sigma(1)\\d(1)\end{array}\right)$,
 and $(\sigma,d) \succ (\tau,e)$ is the sum of the shufflings of these two biwords such that the first biletter is
 $\left(\begin{array}{c} \tau(1)+k\\e(1)\end{array}\right)$. For example:
 \begin{eqnarray*}
 \left(\begin{array}{cc} 1&2\\a&b\end{array}\right)\prec  \left(\begin{array}{cc} 2&1\\c&d\end{array}\right)&=&
  \left(\begin{array}{cccc} 1&2&4&3 \\a&b&c&d\end{array}\right)+  \left(\begin{array}{cccc} 1&4&2&3 \\a&c&b&d\end{array}\right)
  +  \left(\begin{array}{cccc} 1&4&3&2 \\a&c&d&b\end{array}\right),\\
   \left(\begin{array}{cc} 1&2\\a&b\end{array}\right)\succ  \left(\begin{array}{cc} 2&1\\c&d\end{array}\right)&=&
  \left(\begin{array}{cccc} 4&1&2&3 \\c&a&b&d\end{array}\right)+  \left(\begin{array}{cccc} 4&1&3&2 \\c&a&d&b\end{array}\right)
+  \left(\begin{array}{cccc} 4&3&1&2 \\c&d&a&b\end{array}\right).
\end{eqnarray*}

\section{Bidendriform structures on graded permutations} 

\begin{defi} \begin{enumerate}
\item Let $w=(i_1,\ldots,i_k)$ be a word with letters in $\mathbb{N}^*$, all distinct. There exists a unique increasing bijection $f$ from $\{i_1,\ldots,i_k\}$ into $[k]$.
The \emph{standardization} of $w$ is $std(w)=(f(i_1),\ldots,f(i_k))$. It is an element of $\sym_k$.
\item Let $\sigma \in \sym_k$ and $d:[k]\longrightarrow \mathbb{N}^*$. We put:
\begin{eqnarray*}
\Delta_\prec((\sigma,d))&=&\sum_{k=\sigma^{-1}(1)}^{n-1}
\left(\begin{array}{c}std(\sigma(1),\ldots,\sigma(k))\\ d(1),\ldots,d(k)\end{array}\right) 
\otimes \left(\begin{array}{c}std(\sigma(k+1),\ldots,\sigma(n))\\d(k+1),\ldots,d(n)\end{array}\right)\\
\Delta_\succ((\sigma,d))&=&\sum_{k=1}^{\sigma^{-1}(1)-1}
\left(\begin{array}{c}std(\sigma(1),\ldots,\sigma(k))\\ d(1),\ldots,d(k)\end{array}\right) 
\otimes \left(\begin{array}{c}std(\sigma(k+1),\ldots,\sigma(n))\\d(k+1),\ldots,d(n)\end{array}\right)\\
\end{eqnarray*}
This defines two coproducts on the augmentation ideal $\VS_+$ of the dendriform algebra $\VS_+$.
\end{enumerate}\end{defi}
 
 {\bf Example.} Let $a_1,a_2,a_3,a_4 \in \mathbb{N}^*$.
 \begin{eqnarray*}
\Delta_\prec\left(\left(\begin{array}{cccc}3&1&4&2\\a_1&a_2&a_3&a_4\end{array}\right)\right)&=&
 \left(\begin{array}{cc}2&1\\a_1&a_2\end{array}\right)\otimes \left(\begin{array}{cc}2&1\\a_3&a_4\end{array}\right)
 +\left(\begin{array}{ccc}2&1&3\\a_1&a_2&a_3\end{array}\right)\otimes \left(\begin{array}{c}1\\a_4\end{array}\right),\\
  \Delta_\succ\left(\left(\begin{array}{cccc}3&1&4&2\\a_1&a_2&a_3&a_4\end{array}\right)\right)&=&
 \left(\begin{array}{c}1\\a_1\end{array}\right)\otimes \left(\begin{array}{ccc}1&3&2\\a_2&a_3&a_4\end{array}\right).\\
 \end{eqnarray*}
 
 In other words, the coproducts of a biword $\left(\begin{array}{ccc}\sigma(1)&\ldots&\sigma(n)\\ d(1)&\ldots&d(n)\end{array}\right)$
 are given by the cuts of the biword into two parts  and the standardization of the first lines of the two parts of the biword; in $\Delta_\prec$, 
the biletter $\left(\begin{array}{c}1\\d\circ \sigma^{-1}(1)\end{array}\right)$ is in the left part and in $\Delta_\succ$, it is in the right part. \\
 
 {\bf Notations}. For all $x \in \VS_+$, we put:
$$\Delta_\prec(x)=x'_\prec \otimes x''_\prec,\:\Delta_\succ(x)=x'_\succ\otimes x''_\succ,\:\tdelta(x)=\Delta_\prec(x)+\Delta_\succ(x)=x'\otimes x''.$$

\begin{prop} For all $x,y \in \VS_+$:
\begin{eqnarray}\label{E3} \Delta_\prec(x \prec y)&=&x'_\prec \prec y'\otimes x''_\prec \star y''+x\otimes y+x\prec y'\otimes y''+\end{eqnarray}
\begin{eqnarray*}&&x'_\prec \otimes x''_\prec\star y+x'_\prec \prec y\otimes x''_\prec,\end{eqnarray*}
\begin{eqnarray}\label{E4} \Delta_\succ(x\prec y)&=&x'_\succ \prec y'\otimes x''_\prec \star y+x'_\succ\prec y\otimes x''_\succ+x'_\succ \otimes x''_\succ \star y, \end{eqnarray}
\begin{eqnarray}\label{E5} \Delta_\prec(x \succ y)&=&x'_\prec \succ y'\otimes x''_\prec \star y''+x'_\prec \succ y\otimes x''_\prec+x \succ y'\otimes y'',\end{eqnarray}
\begin{eqnarray}\label{E6} \Delta_\succ(x \succ y)&=&x'_\succ \succ y'\otimes x''_\succ \star y''+y\otimes x+y'\otimes x \star y''+\end{eqnarray}
\begin{eqnarray*}x'_\succ \succ y \otimes x''_\succ.\end{eqnarray*}
Consequently, $(\VS,\succ^{op},\prec^{op},\Delta_\succ^{op},\Delta_\prec^{op})$ is a bidendriform bialgebra.
\end{prop}
 
 \begin{proof} These identities are the axioms for dendriform bialgebras, as introduced in \cite{Foissy1}, to which we also refer for the structure results on dendriform bialgebras we will use further on.
We restrict ourselves to the case where $x=(\sigma,d)$ and $y=(\tau,e)$ are two biwords.
 Then $\Delta_\prec(x \prec y)$ is obtained by taking all the shufflings of $x$ and $y$ such the first letter of the result is the first letter of $x$,
 then cutting these words after the letter $\left(\begin{array}{c}1\\d\circ \sigma^{-1}(1)\end{array}\right)$. As a consequence,
 there are biletters of $x$ in the left part of the result. Hence, five case are possible:
 \begin{enumerate}
 \item There are letters of $x$ and $y$ in both parts: this gives the term $x'_\prec \prec y'\otimes x''_\prec \star y''$.
 \item There are letters of $x$ in both parts, and all the letters of $y$ are in the left part: this gives the term $x'_\prec \prec y \otimes x''_\prec$.
  \item There are letters of $x$ in both parts, and all the letters of $y$ are in the right part: this gives the term $x'_\prec \otimes x''_\prec\star y$.
 \item All the letters of $x$ are in the left part, and there are letters of $y$ in both parts: this gives the term $x\prec y'\otimes y''$.
 \item All the letters of $x$ are in the left part, and all the letters of $y$ are in the right part: this gives the term $x\otimes y$.
 \end{enumerate}
 Let us now consider $\Delta_\succ(x \prec y)$. It is obtained by taking all the shufflings of $x$ and $y$ such the first letter of the result is the first letter of $x$,
 then cutting these words before the letter $\left(\begin{array}{c}1\\d\circ \sigma^{-1}(1)\end{array}\right)$. As a consequence,
 there are biletters of $x$ in both parts of the result. Hence, three cases are possible:
 \begin{enumerate}
 \item There are letters of $y$ in both parts of the result: this gives the term $x'_\succ \prec y'\otimes x''_\succ \star y''$.
 \item All the letters of $y$ are in the left part: this gives the term $x'_\succ \prec y\otimes x''_\succ$.
 \item All the letters of $y$ are in the right part: this gives the term $x'_\succ \otimes x''_\succ \star y$.
 \end{enumerate}
We now consider $\Delta_\prec(x \succ y)$.  It is obtained by taking all the shufflings of $x$ and $y$ such the first letter of the result is the first letter of $y$,
 then cutting these words after the letter $\left(\begin{array}{c}1\\d\circ \sigma^{-1}(1)\end{array}\right)$. Consequently, there are letters of $x$ and $y$
 in the left part. So there are three possibilities:
 \begin{enumerate}
 \item There are letters of $x$ and $y$ in both parts of the result: this gives the term $x'_\prec \succ y'\otimes x''_\prec \star y''$.
 \item All the letters of $y$ are in the left part: this gives the term $x'_\prec \succ y\otimes x''_\prec$.
 \item All the letters of $x$ are in the left part: this gives the term $x\succ y' \otimes y''$.
 \end{enumerate}
 We now consider $\Delta_\succ(x \succ y)$.  t is obtained by taking all the shufflings of $x$ and $y$ such the first letter of the result is the first letter of $y$,
 then cutting these words before the letter $\left(\begin{array}{c}1\\d\circ \sigma^{-1}(1)\end{array}\right)$. Consequently, there are letters of $y$ 
 in the left part, and letter of $x$ in the right part. Consequently, four cases are possible.
 \begin{enumerate}
 \item  There are letters of $x$ and $y$ in both parts of the result: this gives the term $x'_\succ \succ y'\otimes x''_\succ \star y''$.
 \item There are letters of $y$ in both parts and all the letters of $x$ are in the right part: this gives the term $y'\otimes x \star y''$.
 \item There are letters of $x$ in both parts and all the letters of $y$ are in the left part: this gives the term $x'_\succ \succ y \otimes x''_\succ$.
 \item All the letters of $x$ are in the right part and all the letters of $y$ are in the left part: this gives the term $y\otimes x$.
 \end{enumerate}
We obtain in this way the four compatibilities (\ref{E3})-(\ref{E6}).  \end{proof} 
 
{\bf Remark.} We define $\Delta:\VS \longrightarrow \VS \otimes \VS$ by $\Delta(x)=x\otimes 1+1\otimes x+\Delta_\prec(x)+\Delta_\succ(x)$ 
for all $x\in \VS_+$, and $\Delta(1)=1\otimes 1$. Then $(\VS,\star,\Delta)$ is a Hopf algebra. \\
 
By the bidendriform rigidity theorem (according to which a bidendriform bialgebra is a free dendriform algebra, see \cite{Foissy1} for details):
 
\begin{cor}
$(\VS,\prec,\succ)$ is, as a dendriform algebra, freely generated by the subspace of dendriform primitive elements $Prim_{Dend}(\VS):
=Ker(\Delta_\prec)\cap Ker(\Delta_\succ)$.
\end{cor}
 
The generating series of $\VS$ as a graded vector space is:
$$R(x)=\left(\sum_{k=0}^\infty k!x^k\right)\circ\left(\sum_{k=1}^\infty x^k\right).$$
 As a consequence of the bidendriform rigidity theorem, the formal series of bidendriform elements of $\VS$ is:
 $$P(x)=\frac{R(x)-1}{R(x)^2}.$$
 Here are the first coefficients of $R(x)$ and $P(x)$:
 $$\begin{array}{c|c|c|c|c|c|c|c|c|c}
 n&1&2&3&4&5&6&7&8&9\\
 \hline r_n&1&3&11&49&261&1\:631&11\:743&95\:901&876\:809\\
 \hline p_n&1&1&2&10&70&550&4\:730&44\:378&45\:646
 \end{array}$$

\section{The dendriform descent algebra}

Recall that the descent algebra $\mathbb{D}$ of a tensor algebra and, more generally, of any graded bialgebra $H$, is the convolution algebra generated by the projections 
on the graded components of $H$ \cite{patJA,reutenauer}. The descent algebra is a graded algebra, and its graded components can be equipped with an internal composition product. These components identify with the classical Solomon descent algebras of type $A$ and form the main building block of noncommutative representation theory.
Motivated by the structural properties of the descent algebra, Fisher's thesis (where the dual notion is introduced and studied in various particular cases) \cite{fisher} and by the Proposition \ref{eqnPi}, which shows that the canonical projection in shuffle algebras belong to the dendrifrom subalgebra of $End(Sh(X))$ generated by the $p_n$:

\begin{defi}
We define $\Descd\subseteq \VS$ as the dendriform subalgebra of $End(Sh(X))$ and $\VS$ generated by the graded projections $p_n:Sh(X)\longrightarrow Sh^n(X)$.
\end{defi}

The algebra $\Descd$ is naturally graded (the degree of $p_n$ is $n$): $\Descd=\bigoplus\limits_{n\in\mathbb N}Descd_n$ and has the completion $\hat{\Descd}=\prod\limits_{n\in\mathbb N}Decsd_n$.
For simplicity, we do not emphasize the distinction between $\Descd$ and its completion when dealing with formal power series such as $p=\sum_np_n$ and will allow us for example  to write abusively $p\in \Descd$.   

Recall that for all $n$:
 $$p_n=\sum_{k=1}^n \sum_{d(1)+\ldots+d(k)=n} \left(\begin{array}{ccc}1&\ldots&k\\d(1)&\ldots&d(k)\end{array}\right).$$
For the classical descent algebra $\mathbb{D}$, a key property is the group-like behavior of the graded projections $p_n$ ($\Delta(p_n)=\sum_{i\leq n}p_i\otimes p_{n-i}$). We show now that this property is inherited, although in a more sophisticated way, in $\Descd$.
 
 \begin{prop}
 For all $n \geq 1$, $\Delta_\succ(p_n)=0$ and $\displaystyle \Delta_\prec(p_n)=\sum_{i=1}^{n-1} p_i \otimes p_{n-i}$.
 \end{prop}
 
 \begin{proof} Clearly, $\Delta_\succ(p_n)=0$. Moreover:
 \begin{eqnarray*}
 \Delta_\prec(p_n)&=&\sum_{k=1}^n \sum_{d(1)+\ldots+d(k)=n}
 \sum_{i=1}^{k-1} \left(\begin{array}{ccc}1&\ldots&i\\d(1)&\ldots&d(i)\end{array}\right)
\otimes \left(\begin{array}{ccc}1&\ldots&k-i\\d(i+1)&\ldots&d(k)\end{array}\right)\\
&=&\sum_{i=1}^{n-1}\left(\sum_{p=1}^i \sum_{d'(1)+\ldots+d'(p)=i}
 \left(\begin{array}{ccc}1&\ldots&p\\d'(1)&\ldots&d'(p)\end{array}\right)\right)\otimes\\
 && \left(\sum_{q=1}^{n-i} \sum_{d''(1)+\ldots+d''(q)=n-i}
 \left(\begin{array}{ccc}1&\ldots&q\\d''(1)&\ldots&d''(q)\end{array}\right)\right)\\
 &=&\sum_{i=1}^{n-1}p_i \otimes p_{n-i}.
\end{eqnarray*} \end{proof}

\begin{cor}\label{descdlibre}
 The dendriform descent algebra $\Descd$ is a sub bidendriform bialgebra of $\VS$. In particular, it is a free dendriform algebra over its dendriform primitive elements.
\end{cor}

\begin{theo}
The family $(\pi_n)_{n\geq 1}$ is a basis of the space of dendriform primitive elements of $\Descd$,
and these elements freely generate $\Descd$ as a dendriform algebra.
\end{theo}

\begin{proof} Recall the notations $p^+=Id-1=\sum\limits_{n\geq 1} p_n$, $p_0=1$. Then $\Delta_\succ(p^+)=0$ 
and 
$\Delta_\prec(p^+)=\tdelta(p^+)=p^+\otimes p^+$, so that $\Delta(Id)=Id\otimes Id$. 

We set $t:=p^{-1}=\sum_{n=0}^\infty (-1)^n (p^+)^n$ and $t^+=t-1$ and get: $\Delta(t)=\Delta(p^{-1})=p^{-1}\otimes p^{-1}=t\otimes t$ or, equivalently,
$\tdelta(t^+)=t^+\otimes t^+$.

We also have $\pi=p^+\prec t$. Since $\Delta_\succ(p^+)=0$, we get $\Delta_\succ(\pi)=0$ and 
$$\Delta_\prec(\pi)=p^+\otimes p^++p^+\prec t^+\otimes p^+t^+ +p^+\otimes t^+ +p^+\prec t^+\otimes t^++p^+\otimes p^+t^+$$
$$+p^+\prec t^+\otimes p^+=p^+\otimes p^++p^+\prec t^+\otimes (-t^+-p^+) +p^+\otimes t^+ +p^+\prec t^+\otimes t^+$$
$$+p^+\otimes (-t^+-p^+)+p^+\prec t^+\otimes p^+=0$$
where we used the identities $pt=(1+p^+)(1+t^+)=1$ and $p^+t^+=-p^+-t^+$. 
So $\pi$ is primitive for both coproducts. Taking its homogeneous component of weight $n$, we obtain that $\pi_n$ is dendriform primitive for all $n$. \\

Recall now that, by the Corollary~\ref{descdlibre}, $\Descd$ is freely generated as a dendriform algebra by the space of its dendriform primitive elements,
and that, by definition, $\Descd$ is generated by at most one generator in each weight. So the homogeneous components of the space of dendriform primitive elements
for the weight are at most one-dimensional. Finally, $(\pi_n)_{n \geq 1}$ is a basis of $Prim_{dend}(\Descd)$. \end{proof}

\begin{cor}
The formal series of $\Descd$ is:
$$\frac{1-x-\sqrt{(1-x)(1-5x)}}{2x}.$$
\end{cor}

\begin{proof} The formal series of $\Descd$ is $\left(\frac{1-\sqrt{1-4x}}{2x}\right)\circ \left(\frac{x}{1-x}\right)$. \end{proof}

{\bf Examples.}
$$\begin{array}{c|c|c|c|c|c|c|c|c|c}
n&1&2&3&4&5&6&7&8&9\\
\hline dim(\Descd_n)&1&3&10&36&137&543&2\:218&9\:285&39\:587
\end{array}$$
This is sequence A002212 of the On-Line Encyclopedia of Integer Sequences. \\

{\bf Remark.}  As a consequence, it is not difficult to prove that the following families are bases of $\Descd_n$:
\begin{enumerate}
\item $n=1$: $(\pi_1)=\left(\left(\begin{array}{c}1\\1\end{array}\right)\right)$.
\item $n=2$: $(\pi_2,\pi_1\prec \pi_1,\pi_1 \succ \pi_1)=\left(\left(\begin{array}{c}1\\2\end{array}\right),
\left(\begin{array}{cc}1&2\\1&1\end{array}\right),\left(\begin{array}{cc}2&1\\1&1\end{array}\right)\right)$.
\item $n=3$: $(\pi_3,\pi_1 \prec \pi_2,\pi_1 \succ \pi_2,\pi_2 \prec \pi_1,\pi_2\succ \pi_1,\pi_1 \prec (\pi_1\prec \pi_1),\pi_1\prec (\pi_1 \succ \pi_1),
\pi_1 \succ (\pi_1 \prec \pi_1), \pi_1 \succ (\pi_1 \succ \pi_1)-(\pi_1 \succ \pi_1) \succ \pi_1,(\pi_1 \succ \pi_1) \succ \pi_1)=$
$$ \left(\begin{array}{c}
\left(\begin{array}{c}1\\3\end{array}\right),\left(\begin{array}{cc}1&2\\1&2\end{array}\right),
\left(\begin{array}{cc}2&1\\2&1\end{array}\right),\left(\begin{array}{cc}1&2\\2&1\end{array}\right),
\left(\begin{array}{cc}2&1\\1&2\end{array}\right),\\
\left(\begin{array}{ccc}1&2&3\\1&1&1\end{array}\right),\left(\begin{array}{ccc}1&3&2\\1&1&1\end{array}\right),
\left(\begin{array}{ccc}2&1&3\\1&1&1\end{array}\right)+\left(\begin{array}{ccc}2&3&1\\1&1&1\end{array}\right),\\
\left(\begin{array}{ccc}3&1&2\\1&1&1\end{array}\right),\left(\begin{array}{ccc}3&2&1\\1&1&1\end{array}\right)\end{array}\right).$$
\end{enumerate}
So $\Descd$ is not stable under the internal product $\circ$: for example,
$\left(\begin{array}{ccc}1&3&2\\1&1&1\end{array}\right)\circ \left(\begin{array}{ccc}3&1&2\\1&1&1\end{array}\right)=
\left(\begin{array}{ccc}2&1&3\\1&1&1\end{array}\right)\notin \Descd$.

\section{Abstract shuffle bialgebras and the rigidity theorem}

The coproduct acting on the shuffle bialgebra $Sh(X)$ satisfies the relation:
\begin{eqnarray}
\label{left} \Delta (x\prec y)&=&x'\prec y'\otimes x''\shuffle y'' + 1\otimes x\prec y
\end{eqnarray}
where we used the Sweedler notation $\Delta(x)=x'\otimes x''$ and set $1\prec 1=0$ in the formula (recall that $1\prec w=w\succ 1=0$ and $w\prec 1=1\succ w=w$ whenever $w$ is a non empty word).
The relation follows from the recursive definition of the shuffle product of words (or from the observation that for two words $w$ and $z$, 
the first letter of $w$ is the first letter of $w\prec z$, so that the expansion of $\Delta (w\prec z)$ always starts with the first letter of $w$).

These identities lead to the abstract definition of a shuffle bialgebra, a particular case of the notion of dendriform bialgebra \cite{Chap1,ronco,Foissy1}. 

Recall from the discussion at the begining of the article that a (non unital) shuffle algebra is, in general, a vector space $V$ equipped with a bilinear map $\prec$ satisfying the axiom (\ref{shuffle}): $(a\prec b)\prec c=a\prec (b\prec c+c\prec b)$. A unital shuffle algebra $A$ is then obtained by adding a unit to $V$: $A=A^+\oplus\CC:=V\oplus \CC$, with $a\prec 1:=a,\ 1\prec a:=0$. The half-product $1\prec 1$ is defined to be $0$.
The shuffle product is defined on $V=A^+$ by $a\shuffle b:=a\prec b+b\prec a$ or $\shuffle = \prec +\succ$ with $a\succ b:=b\prec a$ and extended to $A$ by requiring $1$ to be the unit. It provides $A$ with the structure of a commutative (and associative) algebra with unit. The shuffle algebras over generating sets $Sh(X)$ are the free algebras over $X$ for the relation (\ref{shuffle}) \cite{schutz}. A shuffle algebra $A$ is graded and connected if its decomposition into graded components $A=\bigoplus_nA_n$ is such $A_0=\CC$ and that the half-product is compatible with the grading ($A_n\prec A_m\subset A_{n+m}$).

\begin{defi}\label{asa}
 Let $A=V\oplus \CC$ be a unital shuffle algebra and $\Delta$ a coassociative counital coproduct on $A$. The coproduct defines a shuffle bialgebra structure on $A$ if and only if the relation (\ref{left}) is satisfied. A shuffle bialgebra is graded connected if it is a graded connected shuffle algebra and if the coproduct is compatible with the grading ($\Delta(A_n)\subset \bigoplus\limits_{i+j=n}A_i\otimes A_j$).
\end{defi}

As expected, shuffle bialgebras over finite sets $Sh(X)$ provide examples for this abstract definition of graded connected shuffle bialgebras.

\begin{lemma}
 The set of linear endomorphisms of a shuffle bialgebra $End(A)$ is equipped with the structure of a dendriform algebra by the products:
$$f\prec g(a):= f(a')\prec g(a'')$$
$$f\succ g(a):= f(a')\succ g(a'')=g(a'')\prec f(a').$$
\end{lemma}

The proof follows from the same arguments as for the Lemma~\ref{dendalgle}.

From now on, $A$ will denote an arbitrary graded connected shuffle bialgebra.

\begin{cor}
 There is a unique map $\phi$ of dendriform algebras from $\Descd$ to $End(A)$ such that $\phi(p_n):=a_n$, where we write $a_n$ for the canonical projection from $A$ to $A_n$.
\end{cor}

Indeed, as the $\pi_n$ to which they are related by triangular equations ($\pi_n$ and $p_n$ are equal up to dendriform products of lower degrees elements),
the $p_n$ form a free family of dendriform generators of $\Descd$. The Corollary follows by the universal properties of free algebras.

\begin{lemma}
Let $\tau=\sum_{n\geq 1}\tau_n:=\phi(\pi)$. 
 For any $x,y\in A^+$, we have $\tau(x\prec y)=0$. In other terms, $\tau$ acts trivially on $A^+\prec A^+$.
\end{lemma}

Indeed, let $a^+=\phi(p^+)$. Since $\pi= p^+\prec (Id)^{-1}$, and since the convolution inverse of $Id$ in $End(A)$ is the antipode, we get:
$$\tau(x\prec y)=a^+(x'\prec y')\prec S(y'')\shuffle S(x'').$$
Since $1\prec u=0$ for an arbitrary $u\in A^+$, we get:
$$\tau(x\prec y)=(a^+(x')\prec y')\prec S(y'')\shuffle S(x'')$$
$$=a^+(x')\prec (y'\shuffle S(y''))\shuffle S(x'')=\tau(x)a_0(y)=0,$$
from which the Lemma follows.

\begin{cor}
 The operator $\tau$ is an idempotent: $\tau^2=\tau$.
\end{cor}

Indeed, since $p=\exp^\prec(\pi)$, $Id_A=:a=\exp^\prec(\tau)$ and:
$$\tau=\tau\circ a=\tau(\exp^\prec(\tau))=\tau\circ \tau$$
since the image of $\tau$ in contained in $A^+$, and therefore $\tau\circ (\tau\prec(\tau...(\tau\prec\tau)...)=0$ for iterated products $(\tau\prec(\tau...(\tau\prec\tau)...)$ of an arbitrary length.

\begin{prop}\label{proj}
 The idempotent map $\tau$ is a projection onto the primitive elements of $A$.
\end{prop}

From the identity $\tau =a^+\prec S$ we get that for a primitive element $x$ of $A$, $\tau(x)=x\prec 1=x$.
Now, for $y\in A^+$, 
$$\Delta(\tau(y))=\Delta(a^+(y')\prec S(y''))=\Delta(y'\prec S(y'')) $$
$$=  (y^1\prec S(y_4))\otimes (y_2\shuffle S(y_3)) + 1\otimes y'\prec S(y'')$$
where we used the coassociativity of the coproduct, the property of the antipode $\Delta(S(y))=S(y'')S(y')$, and an extended Sweedler notation to write somehow abusively $y_1\otimes ...\otimes y_n$ for the iterated coproduct of order $n$ of $y$ (so that e.g. $(\Delta\otimes Id)\circ \Delta (y)=y_1\otimes y_2\otimes y_3$, and so on). 

From the coassociativity of the coproduct, we get $y_2\shuffle S(y_3)=1$ (the convolution product of the antipode with the identity is the null map on $A^+$ and the identity on the scalars). Finally, by cancellation of the non scalar terms in $y_2\shuffle S(y_3)$, we get:
$$\Delta(\tau(y))=y'\prec S(y'')\otimes 1+1\otimes y'\prec S(y'')=\tau(y)\otimes 1+1\otimes \tau(y),$$
from which the Proposition follows.

\ \par

In 2000, F. Chapoton introduced the breaking new idea that the classical Cartier-Milnor-Moore theorem holds in fact for generalized bialgebras such as dendriform bialgebras, provided an analogue of the Poincar\'e-Birkhoff-Witt theorem holds \cite{Chap1}. His main Theorem (\cite[Thm 1]{Chap1}) implies a ``rigidity theorem'' in the commutative case: the underlying algebras are then free shuffle algebras.
Chapoton's proof follows by adapting the classical proof of the Cartier-Milnor-Moore theorem \cite{mm} to dendriform bialgebras. M. Ronco contributed by various remarks to the final preprint version of \cite{Chap1} and proposed soon after another proof by adapting the combinatorial proof \cite{p0,Pat93,patJA} of the Cartier-Milnor-Moore theorem \cite{ronco2}.

As far as classical bialgebras are concerned, it is a well-known fact that the Leray theorem (which asserts that a graded connected commutative bialgebra over a field of characteristic zero is a free commutative algebra) is much simpler to prove than the Cartier-Milnor-Moore theorem, see e.g. \cite{patHopf} for a modern general proof and further references on the subject. This observation also holds for dendriform and shuffle algebras: the various proofs of the classical Leray theorem can be adapted to shuffle bialgebras to get simple and direct proofs of Chapoton's rigidity theorem.
We deduce here a proof of the Theorem from the $Descd$ approach.

\begin{theo}
 A graded connected shuffle bialgebra $A$ is isomorphic, as a shuffle bialgebra, to the free shuffle algebra over the vector space of its primitive elements $Prim(A)$.
\end{theo}

From the Proposition~(\ref{proj}) we know that $\tau$ projects to $P:=Prim(A)$. From the identity $Id=\exp^\prec(\tau)$, we deduce that the image of $\tau$, $P$, generates $A$ as a shuffle algebra and that an arbitrary element in $A$ can be written as a linear combination of iterated half-shuffle products $p_1\prec (p_2\prec ...(p_{n-1}\prec p_n)...)$, with $p_i\in P$

Let us choose a graded basis $\B$ of $P$ (by graded we mean that any $b\in \B$ belongs to a graded component $P_n$ of $P$).
To prove the theorem, it is enough to prove that an arbitrary linear combination of iterated half-shuffle products $\sum_{n\leq N}\sum_{p_1,...,p_n} \lambda_{p_1,...,p_n}p_1\prec (p_2\prec ...(p_{n-1}\prec p_n)...)$ with the $p_i$ in $\B$ vanishes if and only if all the coefficients $\lambda_{p_1,...,p_n}$ are null.

Let us prove this property by induction on $N$. We assume therefore that the $p_1\prec (p_2\prec ...(p_{n-1}\prec p_n)...)$, where $n<N-1$ and the $p_i$ run over $\B$ are linearly independent. Assume now that $X=\sum_{n\leq N}\sum_{p_1,...,p_n} \lambda_{p_1,...,p_n}p_1\prec (p_2\prec ...(p_{n-1}\prec p_n)...)=0$
 and that $\exists (p_1,...,p_N), \ \lambda_{p_1,...,p_N}\not=0$. Then, according to eqn~(\ref{left}), $0=\Delta(X)=\lambda_{p_1,...,p_N}p_1\otimes p_2\prec (...(p_{n-1}\prec p_n)...)+Z$, where $Z$ is a linear combination of elements that, by induction, are linearly independent of $p_1\otimes p_2\prec (...(p_{n-1}\prec p_n)...)$. The Theorem follows.

\begin{cor}
 In particular, the dendriform algebra of graded permutations $\VS$ acts naturally on an arbitrary graded connected shuffle bialgebra $A$.
\end{cor}

This follows from the existence of an isomorphism $A\cong Sh(Prim(A))$.



\begin{thebibliography}{100}

\bibitem{Aguiar} M. Aguiar, Infinitesimal bialgebras, pre-Lie and dendriform algebras, in "Hopf Algebras", Lecture Notes in Pure and Applied Mathematics vol 237 (2004) 1-33. 

\bibitem{bh} P. Baumann and C. Hohlweg, A Solomon-type epimorphism for Mantaci-Reutenauer’s algebra
of a wreath product $G \lmoustache S_n$, Trans. Amer. Math. Soc. 360 (2008), 1475-1538.

\bibitem{bp} C. Brouder and F. Patras, Hyperoctahedral Chen calculus for effective Hamiltonians. J. Algebra 322 (2009), 4105-4120.

\bibitem{Chap1} F. Chapoton, Un th\'eor\`eme de Cartier-Milnor-Moore-Quillen pour les big\`ebres dendriformes et les alg\`ebres braces,
Journal of Pure and Applied Algebra, volume 168 no 1 (2002), pages 1-18.

\bibitem{chnt} F. Chapoton, F. Hivert, J.-C. Novelli and J.-Y. Thibon, An operational calculus for the Mould operad, Int. Math. Res. Not. IMRN 2008, no. 9, Art. ID rnn018. 

\bibitem{cuv1} C. Cuvier, Homologie des alg\`ebres de Leibnitz, PhD Thesis, Strasbourg, Jan 1991.

\bibitem{cuv2}  C. Cuvier, Homologie de Leibniz et homologie de Hochschild, Comptes Rendus Acad. Sci., S\'erie 1, Math\'ematique,
1991, vol. 313, no.9, pp. 569-572.

\bibitem{dht} G. Duchamp, F. Hivert and J.-Y. Thibon, Noncommutative symmetric functions VI: free quasi-symmetric functions and related algebras, International Journal of Algebra and Computation 12 (2002), 671-717. 

\bibitem{emp} K. Ebrahimi-Fard, D. Manchon and F. Patras. New identities in dendriform algebras . J. Algebra 320 (2), 708-727, (2008).

\bibitem{km} K. Ebrahimi-Fard and D. Manchon, Dendriform Equations, Journal of Algebra, 322, (2009), 4053-4079.

\bibitem{EM0} S. Eilenberg and S. Mac Lane,
Cohomology theory of abelian groups and homotopy theory III, Proc. Natl Acad. Sci., 37, (1951), pp. 307-310.

\bibitem{EM} S. Eilenberg and S. Mac Lane,
On the Groups H($\pi$, n), The Annals of Mathematics, Second Series, Vol. 58, No. 1 (Jul., 1953), pp. 55-106.

\bibitem{fisher} F. Fisher, Cozinbiel Hopf algebras in combinatorics, PhD, George Washington University, 2010.

\bibitem{Foissy1} L. Foissy, Bidendriform bialgebras, trees, and free quasi-symmetric functions. J.Pure Appl. Algebra 209 (2007), no. 2, 439-459.

\bibitem{GK} V. Ginzburg and M. Kapranov, Koszul duality for operads, Duke Math. J., Volume 76, Number 1 (1994), 203-272. 

\bibitem{lod} J.-L. Loday, Dialgebras, in Dialgebras and Related Operads, F. Chapoton et al. eds  Springer, Lecture Notes in Mathematics, Vol. 1763
2001.

\bibitem{McL1} S. MacLane,
The Homology Products in K(II, n), Proceedings of the American Mathematical Society, Vol. 5, No. 4 (Aug., 1954), pp. 642-651.

\bibitem{MR}{ C. Malvenuto} and { C. Reutenauer}
{\it Duality between Solomon's algebra and quasi-symmetric functions},
J. Algebra, {\bf 177} (1995), 967--982.

\bibitem{mr2} R. Mantaci and C. Reutenauer, A generalization of Solomon’s descent algebra for hyperoctahedral
groups and wreath products, Comm. Algebra 23 (1995), 27-56.

\bibitem{mnt} F. Menous, J.-C. Novelli and J.-Y. Thibon, Mould calculus, polyhedral cones, and characters of combinatorial Hopf algebras, arXiv:1109.1634.

\bibitem{mm} J.W. Milnor, J.C. Moore, On the structure of Hopf algebras, Ann. of Math. 81 (2) (1965) 211-264.

\bibitem{nt} J.-C. Novelli and J.-Y. Thibon, Free quasi-symmetric functions and descent algebras for wreath products, and noncommutative multi-symmetric functions, Discrete Mathematics 310 (2010), 3584-3606. 

\bibitem{nt2} J.-C. Novelli and J.-Y. Thibon, Construction of dendriform trialgebras, C. R. Acad. Sci. Paris Volume 342, 6, 365-446.

\bibitem{patSMF} F. Patras, Construction g\'eom\'etrique des idempotents eul\'eriens. Filtration des groupes de polytopes et des groupes d'homologie de Hochschild. Bull. Soc. math. France, 119, 1991, p. 173-198.

\bibitem{p0} F. Patras. \textit{Homoth\'eties simpliciales}, Th\`ese de doctorat, Paris 7, Jan. 1992.

\bibitem{Pat93}{\sc F. Patras},
{\it La d\'ecomposition en poids des alg\`ebres de Hopf},
Ann. Inst. Fourier, {\bf 43} (1993), 1067--1087.

\bibitem{patJA} F. Patras, L'alg\`ebre des descentes d'une big\`ebre gradu\'ee. J. Algebra 170, 2 (1994), 547-566.

\bibitem{patHopf} F. Patras, A Leray theorem for the generalization to operads of Hopf algebras with divided powers. J. Algebra 218, (1999), 528-542. 

\bibitem{ronco} M. Ronco, Primitive elements in a free dendriform algebra, in: New Trends in Hopf Algebra Theory
(La Falda, 1999), Amer. Math. Soc., Providence, RI, 2000, pp. 245-263.

\bibitem{ronco2}M. Ronco,
Eulerian idempotents
and Milnor-Moore theorem
for certain non-cocommutative Hopf algebras,
 J. Algebra 254 (2002) 152-172

\bibitem{sauzin} D. Sauzin, Mould expansions for the saddle-node and resurgence monomials, in
 Renormalization and Galois Theories, A. Connes et al eds,
IRMA Lectures in Mathematics and Theoretical Physics Vol. 15
2006.
 
\bibitem{schutz}M. P. Sch\"utzenberger, Sur une propri\'et\'e combinatoire des alg\`ebres de Lie libres pouvant \^etre utilis\'ee dans un probl\`eme de math\'ematiques appliqu\'ees, S\'eminaire Dubreil--Jacotin Pisot (Alg\`ebre et th\'eorie des nombres), Paris, Ann\'ee 1958/59.

\bibitem{reutenauer} C. Reutenauer. \textit{Free Lie algebras}. \rm Oxford University Press, 1993.
\end{thebibliography}
\end{document}